\documentclass[12pt,leqno]{amsart}

\usepackage{amsmath,showidx,amscd}
\usepackage{amssymb}
\usepackage{latexsym}

\newtheorem{thm}[subsection]{Theorem}
\newtheorem{pr}[subsection]{Proposition}
\newtheorem{lm}[subsection]{Lemma}
\newtheorem{cor}[subsection]{Corollary}

\newcommand{\sm}{\raisebox{2.33pt}{~\rule{6.4pt}{1.3pt}~}}
\newtheorem{df}[subsection]{Definition}

\theoremstyle{definition}

\numberwithin{equation}{subsection}

\def\Gal{\text{\rm Gal}}
\def\Pic{\text{\rm Pic}}

\def\Z{\mathbf Z}
\def\Q{\mathbf Q}
\def\R{\mathbf R}
\def\C{\mathbf C}
\def\P{\mathbf P}
\def\F{\mathbf F}

\def\n{\noindent}

\input xy \xyoption{all} \CompileMatrices

\begin{document}

\title
{The discriminant
and the determinant \\
of
a hypersurface of 
even dimension}
\author{Takeshi Saito}
\address{Department of Mathematical Sciences, 
University of Tokyo, Tokyo 153-8914, Japan}
\email{t-saito@ms.u-tokyo.ac.jp}
\date{}

\begin{abstract}
For a smooth hypersurface
of even dimension,
the quadratic character
of the absolute Galois
group defined by
the determinant
of the $\ell$-adic cohomology
of middle dimension
is computed via
the square root
of the discriminant
of a defining polynomial
of the hypersurface.
\end{abstract}

\maketitle

Let $k$ be a field, $\bar k$ an algebraic closure of  $k$  and  $k_s$
   the maximal separable extension of  $k$  contained in $\bar k$. Let
$\Gamma_k =  \Gal(k_s/k) = {\rm Aut}_k(\bar k)$.

Let $X$ be a proper
smooth variety
of even dimension $n$
over $k$.
If $\ell$ is a prime number
which is
invertible in $k$,
the $\ell$-adic cohomology
$V=H^n(X_{\bar k},
{\Q}_\ell(\frac n2))$
defines an orthogonal
representation
of the absolute Galois group
$\Gamma_k={\rm Gal}(k_s/k)$.
The determinant
$$\det V : \Gamma_k \to \{\pm1\}
\subset 
{\Q}_\ell^\times$$
is independent of the choice of $\ell$ (see Corollary \ref{corCb}).

Assume that
$X$ is a smooth hypersurface
of degree $d$
in a projective space
of dimension $n+1$,
and let  $f$  be a homogeneous polynomial 
defining it.
Let ${\rm disc}_d(f)$ be the divided
discriminant of $f$ (see  \S\ref{sdi}).
Assume further that
the characteristic of $k$
is not $2$. We shall prove below (Theorem \ref{thmsign}):

\smallskip

\noindent {\bf Theorem}. {\it The quadratic character $\det V $ is defined by the square root of $\varepsilon(n,d)\cdot{\rm disc}_d(f)$, where $\varepsilon(n,d)$
is  $(-1)^{\frac{d-1}{2}}$
if $d$ is odd
and
is $(-1)^{\frac d2
\cdot \frac{n+2}2}$
if $d$ is even.}

\smallskip

\noindent [In other words, the kernel of $\det V : \Gamma_k \to \{\pm1\}$ is the
subgroup of $\Gamma_k$ corresponding to the field extension
$k(\sqrt{\varepsilon(n,d)\cdot{\rm disc}_d(f)}\ )/k.]$

\smallskip

The proof\footnote{For $n=2$, this had already been proved by the late Torsten Ekedahl,
in an unpublished 2-page manuscript dated February 1993. He was then answering a question raised by J-P. Serre in a Coll\` ege de France course on Galois representations and motives. Ekedahl's method was similar, but not identical, to the one in this article. 
The author was happy to be able to discuss
these questions with him in november 2011, a few days before his unexpected
death.} is in two parts. One shows (cf.\ \S\ref{sde}) by a standard argument on universal families that the theorem is true up to a sign depending only on  $d$  and $n$. One then concludes
that this sign 
is equal to $\varepsilon(n,d)$,
using a topological computation given in
\S\ref{scc}.

\smallskip

The author would like to express
his sincere gratitude to Jean-Pierre
Serre for asking the author to compute the determinant
in terms of the discriminant of
a defining polynomial.
He kindly allows the author
to include his topological argument
in Section \ref{scc} and
Examples on cubic curves
and cubic surfaces in Section \ref{sEx}.
The author thanks Serre
and an anonymous referee
for numerous and helpful comments
to improve the presentation of
the article.
The author also thanks Michel Demazure
for showing him a manuscript
of \cite{De} before publication.
The research is partially
supported by
JSPS Grant-in-Aid (A) 22244001.

\section{ Determinant
of the complex conjugation}
\label{scc}

We compute the
determinant of
the action of the complex
conjugation on
the cohomology of
a proper smooth variety
of even dimension
over ${\mathbf R}$.

\subsection{}\label{ssKG}
We put $\Gamma_{\mathbf R}={\rm Gal}
({\C}/
{\R})
=\{1,c\}$.
We identify the Grothendieck group
$K_0({\Q}[\Gamma_{\mathbf R}])$
of representations
of $\Gamma_{\R}$ over
${\Q}$
with the group
$\{(a,b)\in
{\Z}\times {\Z}
\mid
a\equiv b\pmod 2\}$
by the isomorphism
sending the class
$[\rho]$ of
a representation
$\rho$ to $({\rm Tr} \, \rho(1),
{\rm Tr} \, \rho(c))$.
For a representation $V$
of $\Gamma_{\R}$ and an integer $m$,
let $V(m)$
denote the twist of $V$ by
the $m$-th power of
the non-trivial
character 
$\Gamma_{\R}\to \{\pm1\}$.
The automorphism
of $K_0({\Q}[\Gamma_{\R}])$
sending the class $[V]$
to $[V(1)]$ is given by
$(a,b)\mapsto 
(a,-b)$.
The map
$\det\colon
K_0({\Q}[\Gamma_{\R}])
\to \{\pm1\}$
sending the class
$[\rho]$
to $\det \rho(c)$
is given by
$(a,b)
\mapsto
(-1)^{\frac{a-b}2}$.

\begin{pr}\label{prtop}
Let $X$ be a 
separated scheme
of finite type
of dimension $n$ over 
${\R}$
and let
$e_{\mathbf C}$
and
$e_{\R}$
denote the Euler-Poincar\'e
characteristics 
with compact support of
the topological spaces
$X({\mathbf C})$
and $X({\R})$
respectively.
Then, the alternating
sum
$$[H^*_c(X({\mathbf C}),
{\Q})]
=
\sum_{q=0}^{2n}
(-1)^q
[H^q_c(X({\mathbf C}),
{\Q})],$$
viewed as an element of 
$K_0({\Q}[\Gamma_{\bf R}]) \subset \mathbf{Z} \times \mathbf{Z}$,
is equal to
$(e_{\mathbf C},
e_{\R})$.
\end{pr}

\begin{proof}
It suffices to
show that $e_{\R}
={\rm Tr}(c:
H^*_c(X({\mathbf C}),
{\Q}))$.
By the long exact sequence
$$ \cdots \to 
H^*_c(X({\mathbf C}) \! \sm \! X({\R}),
{\Q})
\to 
H^*_c(X({\mathbf C}),
{\Q})
\to 
H^*_c(X({\R}),
{\Q})
\to \cdots  \,,$$
we have
$${\rm Tr}(c \! : \! H^*_c(X({\mathbf C}),\! \Q)) \! = \!
{\rm Tr}(c  \! : \!  H^*_c(X({\R}),\! \Q))
+
{\rm Tr}(c \!: \!H^*_c(X({\mathbf C}) 
\! \sm \!
X({\R}),\! \Q)).$$
Since $X({\R})$
consists of
the $c$-fixed points,
we have
$${\rm Tr}(c:
H^*_c(X({\R}),
{\Q}))
=
e_{\R}.$$

We now show that
${\rm Tr}(c:H^*_c(X({\mathbf C}) \! \sm \! 
X({\R}),
{\Q}))=0$.
By a standard d\'evissage,
we can reduce the question to the case where
$X$ is affine, and then to the case where it is projective.
Hence, we may assume that
$X$ is a closed subscheme
of a projective space
${\P}^m_{\R}$
over ${\R}$.
The space
${\P}^m({\mathbf C})$
can be embedded as
a compact algebraic subset
(\cite[Definition 2.1.1]{RAG}) in
the space of
$m+1$ by $m+1$ Hermitian matrices
by the immersion sending the vector
$z=(z_0,\ldots,z_m)$
to the matrix
$\dfrac{^t\bar z\cdot z}{ \bar z\cdot ^t\! z}$.

\smallskip 

By \cite[Proposition 2.2.7]{RAG}, this implies that the quotient
$X({\mathbf C})
/\Gamma_{\R}$
is homeomorphic
to a compact semi-algebraic subset
(\cite[Definition 2.1.3]{RAG}) of
${\R}^{(m+1)^2}$.
Hence
by \cite[Th\'eor\`eme 9.2.1]{RAG},
there exists a triangulation
of $X({\mathbf C})
/\Gamma_{\R}$
such that
the image of
$X({\R})$
is a union of simplices.
For each simplex $\Delta$
in the triangulation,
the inverse image of
$\Delta \! \sm \! (\Delta
\cap {\rm Image}(X({\mathbf R})))$
in $X({\mathbf C})$ consists of
its 2 copies switched by $c$.
Hence the equality
${\rm Tr}(c:H^*_c(X({\mathbf C}) \! \sm \! 
X({\R}),
{\Q}))=0$
follows.
\end{proof}
\smallskip

\begin{cor}\label{cortop}
Assume that $X$ 
is proper and smooth
of even dimension $n$.

$1$. If we put $N=\frac 12 (e_{\mathbf C} -(-1)^{\frac n2}e_{\R}),$
we have $$\det (c\colon H^n(X({\mathbf C}),
{\Q}(\textstyle{\frac n2})))=(-1)^N.$$ 

$2$. 
Let $X_d$ be
the Fermat hypersurface
in ${\P}^{n+1}_{\R}$
of degree $d$ defined by the equation
$$T_0^d+\cdots+
T_{n+1}^d = 0 .$$
Then $:$
\begin{equation}
\det (c\colon
H^n(X_d({\mathbf C}),
{\Q}({\textstyle{ \frac n2}})))
=
\begin{cases}
(-1)^{\frac{d-1}2} 
&\text{ if $d$ is odd,}\\
(-1)^{\frac d2\frac {n+2}2}
&
\text{ if $d$ is even.}
\end{cases}
\label{eqVdc}
\end{equation}
\end{cor}
\smallskip

\begin{proof}
1. Let $m=n/2$. Recall that ${\Q}(m)$ is the $m$-th twist of $\Q$, as defined in 1.1; this means that the complex conjugation acts  on it
by $(-1)^m$.

We have
$\det (c\colon
H^n(X({\mathbf C}),
{\Q}(m)))
=
\det(c\colon
[H^*(X({\mathbf C}),
{\Q}(m))])$
by  Poincar\'e duality.
By Proposition \ref{prtop},
the class 
$[H^*(X({\mathbf C}),
{\Q}(m))]$
in $K_0({\Q}[\Gamma_{\R}])$ is
$(e_{\mathbf C},
(-1)^{m}
e_{\R})$
and we obtain
$\det(c\colon
[H^*(X({\mathbf C}),
{\Q}(m))])
=(-1)^N$.

2.
If $d$ is odd,
the morphism $
(t_0,\ldots,t_{n+1})
\mapsto
(t_0^d,\ldots,t_{n+1}^d)$
of $X_d$ into $X_1$
induces a homeomorphism
$X_d({\R})\to 
X_1({\R})$;
since
$X_1$ is isomorphic to
${\P}^n$,
we have 
$e_{\R}=1$.
If $d$ is even,
then
$X_d({\R}) = \varnothing$
 and
we have 
$e_{\R}=0$. 

Let us now define a polynomial
$\Phi(T)\in {\Z}[T]$ by 
$$\Phi(T)=
\dfrac1T\cdot
\bigl((1-T)^{n+2} -
(1-(n+2)T)\bigr).$$
A standard computation
using Chern classes
shows that
$e_{\mathbf C}=\Phi(d)$.

\medskip

\n{\footnotesize [For the convenience of the reader,
we recall this computation.
The exact sequences

\quad $0\to \Omega^1_{{\P}^{n+1}}
\to {\mathcal O}_{{\P}^{n+1}}
(-1)^{\oplus n+2}
\to {\mathcal O}_{{\P}^{n+1}}
\to 0$

\n and 

\quad $0\to
{\mathcal O}_{X_d}(-d)
\to
\Omega^1_{{\P}^{n+1}}
\otimes
{\mathcal O}_{X_d}
\to
\Omega^1_{X_d}\to 0$

\smallskip

\n imply that
the total Chern class
$c(\Omega^1_{X_d})$ is equal to

\smallskip 

\quad $c({\mathcal O}(-1))^{n+2}
\cdot c({\mathcal O}(-d))^{-1}\cdot [X_d]
=(1-h)^{n+2}\cdot (1-dh)^{-1}\cdot dh$,

\n where $h=c_1({\mathcal O}(1))$ denotes the
class of a hyperplane.
The coefficient of $h^{n+1}$ in
$(-1)^nc(\Omega^1_{X_d})$
is
$(-1)^n\sum_{i=0}^n
\binom{n+2}i (-1)^i
d^{n-i+1}
=
\sum_{j=2}^{n+2}
\binom{n+2}j (-1)^j
d^{j-1}$
and is equal to
$\frac1d\cdot
\bigl((1-d)^{n+2} -
(1-(n+2)d)\bigr)=\Phi(d)$.]
}

\medskip

\n We have
$$\Phi(d)\equiv \Phi(a)
+(d-a)\Phi'(a) \!\! \pmod 4 \ \
\text{ if } \ d\equiv a \!\!  \pmod 2.$$
Since $\Phi(1)=n+1$ and
$\Phi(1)+\Phi'(1)=n+2$,
we have 
$\Phi'(1)=1.$  Hence, if $d$ is odd,
we get $\Phi(d) \equiv n+1 + d - 1 
\equiv (-1)^{n/2}+d-1 \pmod 4$,
and 

\quad  $2N \equiv (-1)^{n/2}+d-1 - (-1)^{n/2} \equiv d-1 \pmod 4$.

\noindent If $d$ is even, since $\Phi(0) = 0$ and $\Phi'(0) =$ $ \binom{n+2}2$ $\equiv 1 + \frac{n}{2} \pmod 2$, we have 

 \quad $2N = \Phi(d) \equiv  d(1 + \frac{n}{2}) \pmod 4$. 
 
 \smallskip
 
\n Hence assertion 2 follows from assertion 1.  \end{proof}

\section{Discriminant}\label{sdi}

The literature contains several non-equivalent definitions of ``the''
discriminant of a homogeneous polynomial. For instance, the discriminant of  $x^2 + y^2$ is sometimes defined as $1$, sometimes as $4$ and sometimes as $-4$. In what follows we shall put indices to the symbol ``disc'' in order to clarify the conventions we use.

   We start with the most standard definition (see e.g. \cite [Chap. 13]{GKZ}),
   which is satisfactory over $\C$, but which is not so over an arbitrary field.
\subsection {\it The discriminant, defined as the resultant
of the partial derivatives.}\label{ssDde}
We fix integers $n\geqslant 0$ 
and $d >1$.
We consider the polynomial ring
${\Z}[T_0,\ldots,T_{n+1}]$
and the free ${\Z}$-module
$E=\bigoplus_{i=0}^{n+1}
\Z\cdot T_i$.
The $d$-th symmetric power
$S^d E$ 
is identified with the free ${\Z}$-module
of finite rank
consisting of
homogeneous polynomials
of degree $d$
in ${\Z}[T_0,\ldots,T_{n+1}]$.
If 
$I=(i_0,\ldots,i_{n+1})
\in {\mathbf N}^{n+2}$ is a
multi-index,
we put

\smallskip

 \quad $T^I=T_0^{i_0}\cdots
T_{n+1}^{i_{n+1}}
\in {\Z}
[T_0,\ldots,T_{n+1}]$ \ \
and  \ \ $|I|=i_0+
\cdots+i_{n+1}$.

\smallskip

\n
The monomials 
$T^I$ of degree $|I|=d$ 
form a basis of 
$S^dE$.
Let 
$(C_I)_{|I|=d}$ 
be the dual basis of $(S^dE)^\vee$
and define
the universal polynomial
$F=\sum_{|I|=d}C_IT^I$.

We consider the resultant\footnote{For the definition of the resultant of $m$
homogeneous polynomials in $m$ indeterminates, see e.g. \cite[Chap.13, \S1.A]{GKZ} and  \cite[n$^{\rm o}$ 4 D\'efinition 3]{De}.}
$${\rm res}(D_0F,\ldots,
D_{n+1}F)$$
of its partial derivatives
$D_0F,\ldots,D_{n+1}F$.
It is a homogeneous polynomial
of degree $m=
(n+2)(d-1)^{n+1}$
in $(C_I)_{|I|=d}$
with integral coefficients.
If we put
\begin{equation}
a(n,d)=\dfrac
{(d-1)^{n+2}-(-1)^{n+2}}d,
\label{eqadn}
\end{equation}
the greatest common divisor
of the coefficients is $d^{a(n,d)}$
by \cite[n$^{\rm o}$ 5 Lemme 11]{De} and 
\cite[Chap.\ 13.1.D Proposition 1.7]{GKZ}.

\begin{df}\label{dfDis}
We call  ${\rm res}(
D_0F,\ldots,
D_{n+1}F)$
the resultant-discriminant of $F$
and we denote it
by ${\rm disc}_r(F)$.
We call 
$${\rm disc}_d(F)
=
\dfrac1{d^{a(n,d)}}
{\rm disc}_r(F)$$
the divided discriminant
of ~$F$.
\end{df}

The divided discriminant
${\rm disc}_d(F)$
is known to be geometrically irreducible
cf.\ \cite[n$^{\rm o}$ 6 Proposition 14]{De} and 
\cite[Chap.\ 13.1.D]{GKZ}.

\smallskip
 By specialization, this gives a meaning to ${\rm disc}_r(f)$ and  ${\rm disc}_d(f)$ for every homogeneous polynomial  $f$ in $n+2$ variables over a commutative ring $R$. 
The divided discriminant satisfies
the following smoothness criterion
due to Demazure:
\begin{pr}[{\rm \cite[n$^{\rm o}$ 5 Proposition 12]{De}}]
\label{prDem}
Let $f$ be a homogeneous polynomial of degree $d$ in $n+2$ variables 
over a commutative ring $R$.
Then, the divided discriminant ${\rm disc}_d(f)$ is
  invertible in $R$ if and only if the corresponding hypersurface is a smooth
  divisor of the projective space $\mathbf {P}^{n+1}_R
  ={\rm Proj}\ R[T_0,\ldots,T_{n+1}]$
 over $R$.
\end{pr}
This smoothness criterion would not work with the resultant-discriminant when $d$ is not invertible in $R$.


\smallskip
  
    The transformation properties of 
 ${\rm disc}_d$ are the same as those of ${\rm disc}_r$, namely (cf.\ 
\cite[n$^{\rm o}$ 5 Proposition 11]{De} and 
\cite{GKZ}):

\smallskip

(2.2.3) \quad ${\rm disc}_d(\lambda f) = \lambda^{(n+2)(d-1)^{n+1}} {\rm disc}_d(f)$ 

\n and 

(2.2.4) \quad  ${\rm disc}_d(f_A) =  \det(A)^{d(d-1)^{n+1}} {\rm disc}_d(f)$,

\smallskip

\noindent where $\lambda$ is any element of $R$, $A$ is any $(n+2)\times(n+2)$-matrix with coefficients in $R$, and $f_A(x) = f(Ax)$. 
The formula about $f_A$, applied to a permutation, shows that ${\rm disc}_d(f)$ does not depend on the indexing of the coordinates. 

\subsection{\it The universal family and the discriminant}\label{ssuniv}
We define the universal family
of hypersurfaces.
We put
${\P}^{n+1}_{\Z}
={\P}(E)
={\rm Proj}\
{\Z}[T_0,\ldots,T_{n+1}]$.
Let 
$\P_{n,d} =
{\P}((S^dE)^\vee)$
be the projective space 
defined by the
dual $(S^dE)^\vee
= {\rm Hom}(S^dE,{\Z})$; it is
the {\it moduli space}
of hypersurfaces of degree
$d$ in ${\P}^{n+1}$; we shall usually write it  $\P$ instead of $\P_{n,d}$.
The universal hypersurface
$$X\subset 
{\P}^{n+1}_{\Z}
\times \P$$
is then defined by
the universal polynomial
$F=\sum_{|I|=d}C_IT^I$.

Let $\Delta$ be
the closed subscheme
of $X$
defined by the vanishing of the
partial derivatives
$D_0F,\ldots,$
$D_{n+1}F$
of the universal polynomial
$F$.
By the Jacobian criterion, 
$X\sm \Delta$
is the maximum open
subscheme of $X$ on which
the canonical morphism
$X\to \P$ is smooth.

For an integer $m\ge 0$,
we identify the ${\Z}$-module
$\Gamma(\P,{\mathcal O}(m))$
with the symmetric power
$S^m((S^dE)^\vee)$
consisting of homogeneous polynomials
in $(C_I)_{|I|=d}$ 
of degree $m$.
The closed subscheme
$D$ of $\P=\P_{n,d}$
defined by 
the divided discriminant
${\rm disc}_d(F)$
is reduced and is
an effective Cartier divisor
flat over ${\Z}$.
For $m=
(n+2)(d-1)^{n+1}$,
the divided discriminant
${\rm disc}_d(F)$
is a basis of the free
${\Z}$-module
$L=\Gamma(\P,
{\mathcal O}(m)(-D))$
of rank 1.

The following is a geometric interpretation
of the smoothness criterion Proposition \ref{prDem}.

\begin{lm}\label{lmgeo}
The underlying set of the reduced closed
subscheme $D$ of $\P$ equals
the image of $\Delta$ by the projection map
${\P}^{n+1}_{\Z} \times \P \ \to \P$.
\end{lm}

\subsection{\it The discriminant and the dual variety}\label{ssdual}
We recall the formalism of
the dual variety \cite[3.1]{SGA7}
and the relation with
the discriminant.

In this subsection,
we fix an algebraically closed field $k$
and let the suffix $_k$ denote
the base change to $k$
over $\Z$.
The indeterminates
$T_0,\ldots,T_{n+1}$
define a homogeneous coordinate
of the projective space
$\P_k^{n+1}=\P(E_k)$.
We put $r=\dim (S^dE_k)-1$
and ${\P}^{r}
={\P}(S^dE_k)$.
The projective space
$\P_k=
{\P}((S^dE_k)^\vee)$
is the dual
$\check{\P}^{r}$ of
${\P}^{r}$ parametrizing hyperplanes
in ${\P}^{r}$.
The monomials
$(T^I)_{|I|=d}$ form a
homogeneous coordinate
of ${\P}^{r}$
and 
the universal coefficients
$(C_I)_{|I|=d}$ form a
homogeneous coordinate
of $\check {\P}^{r}=\P_k$.

The $d$-th symmetric power
$S^dE_k$
is naturally identified
with the space of global sections
$\Gamma({\P}^{n+1}_k,
{\mathcal O}(d))$
of the very ample invertible sheaf
${\mathcal O}(d)$ on
${\P}^{n+1}_k$.
The closed immersion
$$v_d\colon
{\P}^{n+1}_k=
{\P}(E_k)\to
{\P}^{r}
={\P}(S^dE_k)$$
defined by
${\mathcal O}(d)$ on
${\P}^{n+1}_k$ is 
called the $d$-th Veronese embedding.
We consider 
${\P}^{n+1}_k$
as a closed subscheme
of ${\P}^{r}$
by $v_d$.

The conormal sheaf
${\mathcal N}$
of ${\P}^{n+1}_k$
in ${\P}^{r}$
is defined as ${\mathcal I}/{\mathcal I}^2$
where ${\mathcal I}
\subset {\mathcal O}_{{\P}^{r}}$
denotes the ideal sheaf
defining the closed subscheme
${\P}^{n+1}_k\subset 
{\P}^{r}$.
It is a locally free 
${\mathcal O}_{{\P}^{n+1}_k}$-module of rank
$r-(n+1)$.
Let
${\mathbb P}({\mathcal N})
={\rm Proj}\ S^\bullet {\mathcal N}^\vee$ 
denote the associated 
covariant projective space bundle over 
${\P}^{n+1}_k$
{\footnotesize 
[here,
to be consistent with \cite[3.1]{SGA7},
we use the font ${\mathbb P}$
to denote the {\it covariant} projective
space bundle associated
to a locally free sheaf,
parametrizing the lines] }.

We view the projective space bundle
${\mathbb P}({\mathcal N})$ 
over ${\P}^{n+1}_k$ as a closed
subscheme of
${\P}^{n+1}_k
\times\P_k$
as in \cite[(3.1.1)]{SGA7}.
We recall the definition of the embedding
\begin{equation}
{\mathbb P}({\mathcal N})\to
{\P}^{n+1}_k
\times\P_k.
\label{eqemb}
\end{equation}
The projective space 
$\P_k={\mathbf P}((S^dE_k)^\vee)$ is 
the covariant projective space
${\mathbb P}(S^dE_k)$  
associated to the vector space $S^dE_k$.
Since the twist by an invertible
sheaf does not change
the associated projective space bundle,
the product
${\P}^{n+1}_k
\times\P_k$
is identified with
the covariant projective space bundle
${\mathbb P}(S^dE_k\otimes 
v_d^*{\mathcal O}_
{{\P}^r}(-1))$
associated to 
a locally free ${\mathcal O}_
{{\P}^{n+1}_k}$-module
$S^dE_k\otimes 
v_d^*{\mathcal O}_
{{\P}^r}(-1)$.
The canonical map
${\mathcal N}\to \Omega^1_{\P^r}
\otimes {\mathcal O}_
{{\P}^{n+1}_k}$
defined by the closed immersion
${\P}^{n+1}_k
\to \P^r$
and the canonical map
$\Omega^1_{\P^r}
\to S^dE_k\otimes {\mathcal O}_
{{\P}^r}(-1)$
for the projective space
$\P^r={\P}(S^dE_k)$
are locally splitting injections.
Hence the composition
${\mathcal N}\to \Omega^1_{\P^r}
\otimes {\mathcal O}_
{{\P}^{n+1}_k}
\to 
S^dE_k\otimes 
v_d^*{\mathcal O}_
{{\P}^r}(-1)$ defines
the desired closed immersion
${\mathbb P}({\mathcal N})
\to {\P}^{n+1}_k
\times\P_k$
by the covariant functoriality.

The map
$\varphi\colon
{\mathbb P}({\mathcal N})\to
\check{\P}^r$
\cite[(3.1.2)]{SGA7}
is defined as the composition
\begin{equation}
{\mathbb P}({\mathcal N})
\to {\P}^{n+1}_k
\times
\P_k
\to \P_k
\label{eqphi}
\end{equation}
of the immersion (\ref{eqemb})
and the projection.
The dual variety
of ${\P}^{n+1}_k
\subset
\P^r_k$
is defined as the image of $\varphi$
\cite[(3.1.3)]{SGA7}.

\begin{lm}\label{lmDel}
Let $k$ be an algebraically closed field.

{\rm 1.}
We have an equality
${\mathbb P}({\mathcal N})
=\Delta_k$
of the underlying sets
of closed subschemes
of ${\P}^{n+1}_k\times
\P_k$. 

{\rm 2.}
The dual variety of
${\P}^{n+1}_k$
in ${\P}^{r}_k$
defined as the image of
the map
$\varphi\colon
{\mathbb P}({\mathcal N})
\to \P_k$ {\rm (\ref{eqphi})}
equals the underlying set of 
$D_k$.
\end{lm}

\begin{proof}
1.
A $k$-rational point of
$\P(k)$ corresponds to
a hyperplane $H$ in $\P^r$
defined over $k$ since $\P_k$ is the dual of
$\P^r$.
The intersection of 
$H$ with $\P^{n+1}_k$
embedded in $\P^r$
by the $d$-th Veronese embedding
$v_d\colon \P^{n+1}_k
\to \P^r$
is a hypersurface of
degree $d$ and is the
geometric fiber
of the universal family $X\to \P$
at the geometric point
$[H]\in \P(k)$. 
Consequently,
the set of $k$-valued
points $X(k)\subset
{\P}^{n+1}(k)
\times
\P(k)$
consists of the pairs $(x,H)\in
{\P}^{n+1}(k)\times \P(k)$,
where $H$ is
a hyperplane of ${\P}^{r}$ 
containing $x\in
{\P}^{n+1}(k)\subset {\P}^{r}(k)$.

By  \cite[(3.1.1)]{SGA7},
the set of $k$-valued points 
${\mathbb P}({\mathcal N})(k)
\subset{\P}^{n+1}(k)\times \P(k)$
consists of the pairs $(x,H)\in
X(k)$
such that $H$ is tangent
to ${\P}^{n+1}_k$
at $x$.
In other words,
it consists of the points
of $X(k) \subset
{\P}^{n+1}(k)
\times
\P(k)$ where
the geometric fiber
$X_k\to \P_k$
of the canonical map
of the universal family
$X\to \P$ is not smooth.
Thus the assertion follows
by the Jacobian criterion.

2.
Since the dual variety
is defined as the
image $\varphi({\mathbb P}({\mathcal N}))$,
it follows from assertion 1
and the smoothness criterion
Proposition \ref{prDem}.
\end{proof}

\begin{pr}\label{prdual}
$\Delta$ is a projective space
bundle over ${\P}^{n+1}_{\Z}$.
\end{pr}

\begin{proof}
The closed subscheme $\Delta$
of ${\P}^{n+1}_{\Z}
\times \P$ is defined,
on the inverse image of
the open subscheme $D(T_i)$
of ${\P}^{n+1}_{\Z}$,
by the vanishing of the $n+2$ linear forms
$F,
D_0F,\ldots, D_{i-1}F,
D_{i+1}F,\ldots, D_{n+1}F$ in 
$(C_I)_{|I|=d}$
by the equality
\begin{equation}
d\cdot F
=T_0\cdot D_0F+
\cdots+T_{n+1}\cdot D_{n+1}F.
\label{eqEu}
\end{equation}
Hence, it suffices to show that
for every geometric point
of each $D(T_i)$,
these forms are linearly independent.
Or equivalently,
it suffices to show that
for every geometric point
${\rm Spec}\ k \to {\P}^{n+1}_{\Z}$,
the geometric fiber
of $\Delta_k \to
{\P}^{n+1}_k$
is a linear subspace of
$\P_k$
of codimension $n+2$.

Let $k$ be an algebraically closed field
and take the notation in \ref{ssdual}.
By Lemma \ref{lmDel}.1,
we have an equality
$\Delta_k=
{\mathbb P}({\mathcal N})$
of the underlying sets.
By the definition of
the closed immersion 
${\mathbb P}({\mathcal N})
\to 
{\P}^{n+1}_k
\times\P_k$ (\ref{eqemb})
recalled above,
${\mathbb P}({\mathcal N})$
is a linear subbundle of
the projective space bundle
${\P}^{n+1}_k
\times\P_k$
over
${\P}^{n+1}_k$.
Hence the fiber
of $\Delta_k=
{\mathbb P}({\mathcal N})
\to {\P}^{n+1}_k$
is a linear subspace of
$\P_k$ of codimension 
$\dim ({\P}^{n+1}_k\times
\P_k)
-\dim {\mathbb P}({\mathcal N})$.

Since
${\rm rank}\ {\mathcal N}=
\dim \P_k-
\dim {\P}^{n+1}_k$,
the codimension
$\dim ({\P}^{n+1}_k\times
\P_k)
-\dim {\mathbb P}({\mathcal N})$
is equal to
$(\dim \P_k
+\dim {\P}^{n+1}_k)
-(\dim \P_k-
\dim {\P}^{n+1}_k
-1+\dim {\P}^{n+1}_k)
=n+2$.
Thus, the assertion follows.
\end{proof}

\subsection {\it The irreducibility
of the discriminant modulo $p$}\label{ssD}

Let $\pi\colon X\to \P$
denote the canonical morphism
of the universal family of
hypersurfaces.
By  \cite[Corollaire 1.3.4]{PL} there exists an open subset
$W$ of $\Delta$
consisting of the points
$w$ such that
$w$ is an ordinary quadratic singularity
in the fiber $X_{\pi(w)}$.

\begin{lm}\label{lmSGA7}
{\rm 1.}
For every algebraically closed field $k$,
the geometric fiber
$W_k$
is dense in
${\mathbb P}({\mathcal N})=\Delta_k$.

{\rm 2.}
Assume that $n$ is odd or 
${\rm char}\ k\neq 2$.
Let  $W'_k$
be the subset of $W_k\subset 
{\mathbb P}({\mathcal N})$
consisting of the images of 
geometric points $w$
of $W_k$
that is a unique singular
point in the geometric fiber $X_{\pi(w)}$.

Then, $W'_k$
is the maximum open subscheme
of ${\mathbb P}({\mathcal N})$
where the restriction 
of $\varphi\colon
{\mathbb P}({\mathcal N})\to 
\P_k$
is an immersion.
Consequently,
the canonical morphism
${\mathbb P}({\mathcal N})
\to D_{k,{\rm red}}$ to
the maximum reduced subscheme of
$D_k$ is birational. 
\end{lm}

\begin{proof}
1.
By \cite[(3.7.1)]{SGA7},
the open subscheme 
$W_k$ of
${\mathbb P}({\mathcal N})$
is non-empty.
Since
${\mathbb P}({\mathcal N})$
is irreducible,
it is dense.

2.
By assertion 1
and  \cite[Proposition 3.3]{SGA7},
the morphism
$\varphi\colon {\mathbb P}({\mathcal N})
\to
\P_k$ (\ref{eqphi}) is generically
unramified (``unramified'' as defined in
EGA IV (17.3.1) is the same as ``net'' in 
SGA 7 : the diagonal map is an open immersion).
Hence the assertion
follows from \cite[Proposition 3.5]{SGA7}.
\end{proof}

We deduce 
the irreducibility
of the reduction of the discriminant modulo $p$
from Proposition \ref{prdual}
using the following Lemma.

\begin{lm}\label{lmimm}
Let $f\colon X\to Y$
be a proper morphism of
noetherian schemes.
For a point $y$ of $Y$,
let $f_y\colon X_y\to y$ denote the
base change
by the canonical map
$y={\rm Spec}\ \kappa(y)\to Y$
and define a subset $V$
of $Y$ by
$$V=\{y\in Y\mid X_y=\emptyset
\text{ 
or $f_y\colon X_y\to y$ is an isomorphism}
\}.$$
Then
$V$ is the largest open subset
of $Y$ such that
$X\times_YV\to V$
is a closed immersion.
\end{lm}

\begin{proof}
The subset 
$$U=\{x\in X\mid x
\text{ is isolated in }f^{-1}(f(x))\}$$
of $X$ is the largest open subset of $X$
such that the restriction
$f|_U\colon U\to Y$
is quasi-finite by \cite[Proposition (4.4.1)]{EGA3}.
The complement
$W=Y\sm f(X\sm U)$
is the largest open subset of $Y$
such that the base change
$W\times_YX\to W$
is finite by \cite[Proposition (4.4.2)]{EGA3}.
Since $V$ is a subset of $W$,
by replacing $Y$ by $W$,
we may assume $f$ is finite.

By Nakayama's lemma,
the subset $V$ of $Y$
is the complement of
the support of the coherent 
${\mathcal O}_Y$-module
${\rm Coker}({\mathcal O}_Y
\to f_*{\mathcal O}_X)$.
Hence $V$ is the largest
open subscheme where
the morphism
${\mathcal O}_Y
\to f_*{\mathcal O}_X$
is a surjection.
\end{proof}

\begin{pr}\label{prirr}
Unless $p=2$ and $n$ is even,
the divided discriminant
${\rm disc}_d(F)$ modulo $p$
is a geometrically irreducible  polynomial in the $C_I$.
\end{pr}

\begin{proof}
Let $V$
be the largest open subset of $\P$
such that
$\Delta\times_{\P}V\to V$
is a closed immersion.
Since $\Delta$ is reduced by Proposition \ref{prdual},
the induced map $\Delta\times_{\P}V\to
D\times_{\P}V$ is
an isomorphism by Lemma \ref{lmgeo}.

Lemmas \ref{lmSGA7}.2 and \ref{lmimm} imply
that the inverse image of
$V_{{\F}_p}$ 
is dense in
$\Delta_{{\F}_p}$
if $n$ is odd or $p\neq 2$.
Since the geometric fiber
$\Delta_{\overline{{\F}}_p}$ is reduced 
and irreducible by Proposition \ref{prdual},
the open subscheme
$(D\times_{\P}V)_{\overline{{\F}}_p}$ of
$D_{\overline{{\F}}_p}$
is reduced, irreducible and dense.
Hence, the fiber
$D_{{\F}_p}$
is a divisor of
$\P_{{\F}_p}$
defined by a geometrically irreducible
polynomial.
\end{proof}

\smallskip
Proposition \ref{prirr}
is proved in \cite[n$^{\rm o}$ 5 Proposition 14]{De}
under the assumption $p\nmid d(d-1)$.
In the exceptional case
$p=2$ and $n$ even,
we show in Theorem \ref{thm4} that the {\it signed discriminant}
$\varepsilon(n,d){\rm disc}_d(F)$ is 
congruent to a square modulo $4$
where $\varepsilon(n,d)$ is defined as in
Theorem stated in the introduction.

\section{Determinant}\label{sde}

In this section,
we assume that $n$ is even.

\subsection{}\label{ssdet}
Let $S$ be a normal integral
scheme of finite type
over ${\Z}$ and
$f\colon X\to S$ be
a proper smooth morphism 
of relative dimension $n$.
For a prime number
$\ell$ invertible in
the function field of $S$,
the cup-product
defines a non-degenerate
symmetric bilinear form on the
smooth ${\Q}_\ell$-sheaf
$R^nf_*
{\Q}_\ell(\frac n2)$
on $S[\frac1\ell]$.
Hence the determinant
defines
a character
$\pi_1(S[\frac1\ell])^{ab}
\to\{\pm 1\}\subset 
{\Q}_\ell^\times$
of the fundamental group,
which we denote by  $[\det H^n_\ell(X)]$.

\begin{lm}\label{lmDe}
There exists a unique character
$[\det H^n(X)]\colon
\pi_1(S)^{ab}
\to\{\pm 1\}$
such that, for every prime number
$\ell$ invertible in
the function field of $S$,
the composition 
with the map
$\pi_1(S[\frac1\ell])^{ab}
\to \pi_1(S)^{ab}$
induced by the open immersion
$S[\frac1\ell]\to S$ gives
$[\det H^n_\ell(X)]$.
\end{lm}

\begin{proof}
First, we show the
case where $S={\rm Spec}\ k$
for a finite field $k={\F}_q$.
Then, for a prime $\ell$
different form the characteristic of $k$,
Deligne's theorem on $\ell$-adic representations
tells us that
the polynomial 
$\det(1-Fr_qt:
H^n(X_{\bar \F_q},{\Q}_\ell))$ is
in $\Z[t]$ and is independent of $\ell$.
Hence, the multiplicity $m$ of the eigenvalue
of $-q^{n/2}$ of $Fr_q$
acting on 
$H^n(X_{\bar \F_q},{\Q}_{\ell})$
and the determinant
$\det(Fr_q:
H^n(X_{\bar \F_q},{\Q}_{\ell}(\frac n2)))
=(-1)^m$ are independent of $\ell$.
Since the geometric Frobenius
$Fr_q$ generates
a dense subgroup of
$\Gamma_k$,
the assertion follows in this case.

We prove the general case.
For two different primes $\ell$ and $\ell'$
and for every closed point $x$ of
$S[\frac1{\ell \ell'}]$,
the case proved above
gives us a commutative diagram
\begin{equation}\begin{array}{ccccc}
&&\hspace{-1mm}
\pi_1(S[\frac1\ell])^{ab}&&\\
&\hspace{-7mm}
  \nearrow&&
  \hspace{-1mm}
\searrow &\\\\
\Gamma_{\kappa(x)}
&&&&
\hspace{-3mm}
\{\pm 1\}\\\\
&\hspace{-7mm}
\searrow&&
\hspace{-1mm}
\nearrow &\\
&&\hspace{-1mm}
\pi_1(S[\frac1{\ell'}])^{ab}&&
\end{array}
\label{eqCb}
\end{equation}
where the left slant arrows
are induced by the closed immersion of $x$.

By the Chebotarev density theorem 
\cite[Theorem 7]{ZL},
\cite[Theorem 9.11]{STai}
for a normal scheme of finite type
over $\Z$,
the images $\Gamma_{\kappa(x)}
\to\pi_1(S[\frac1{\ell \ell'}])^{ab}$
where $x$ runs the closed points
of $S[\frac1{\ell \ell'}]$ generate
a dense subgroup.
Hence we obtain
a commutative diagram
\begin{equation}
\begin{array}{ccccc}
&&\hspace{-1mm}
\pi_1(S[\frac1\ell])^{ab}&&\\
&\hspace{-7mm}
  \nearrow&&
  \hspace{-1mm}
\searrow &\\\\
\pi_1(S[\frac1{\ell \ell'}])^{ab}
&&&&
\hspace{-3mm}
\{\pm 1\}.\\\\
&\hspace{-7mm}
\searrow&&
\hspace{-1mm}
\nearrow &\\
&&\hspace{-1mm}
\pi_1(S[\frac1{\ell'}])^{ab}&&
\end{array}
\label{eqCbS}
\end{equation}
Glueing of coverings shows
that $\pi_1(S)^{ab}$
is the push-out of the
left half of (\ref{eqCbS}).
Thus, the characters
$[\det H^n_\ell(X)]$
induce a unique character 
$\pi_1(S)^{ab}\to \{\pm1\}$, which is
independent of $\ell$.
\end{proof}

\begin{cor}\label{corCb}
Let $X$ be a proper smooth
scheme of even dimension $n$
over a field $k$.
Then, for a prime number $\ell$
invertible in $k$,
the character
$\det H^n(X_{\bar k},
{\Q}_\ell(\frac n2))$ of
$\Gamma_k$
is independent of $\ell$.
\end{cor}

\begin{proof}
We may assume that $k$
is finitely generated over a prime field.
Hence, there exist
a normal integral scheme $S$
of finite type over ${\Z}$
such that $k$ is the function field of $S$
and a proper smooth morphism
$f\colon X_S\to S$
of relative dimension $n$
such that the generic fiber
is $X$.
Now the assertion follows form
Lemma  \ref{lmDe}.
\end{proof}

\subsection{}\label{ssmu2}
Let $d > 1$
be an integer
and $\pi\colon X
\subset {\P}^{n+1}
_{\Z}\times \P
\to \P
={\P}
((S^dE)^\vee)$
denote the universal
hypersurface of degree $d$
defined by
the universal polynomial $F$.
The complement $U=\P\! \sm \! D$
of the divisor $D$
defined by the divided discriminant
${\rm disc}_d(F)$
is the maximum open
subscheme of $\P$ over which
$\pi\colon X\to \P$
is smooth.

Since $n$ is 
even, the
degree $m=(n+2)(d-1)^{n+1}$
of the discriminant
${\rm disc}_d(F)$
is even.
We consider
the $\mu_2$-torsor
on the fppf-site of $U$
consisting of bases of 
${\mathcal O}(\frac m2)$
whose square is
equal to the basis
${\rm disc}_d(F)$
of an invertible ${\mathcal O}_U$-module
${\mathcal O}_U(m)$.
We denote by $[{\rm disc}_d(F)]$
the class of this torsor in 
$H^1_{\rm fl}(U,\mu_2)$.

By applying Lemma \ref{lmDe}
to the universal smooth hypersurface 
$$\pi_U\colon
X_U\to U,$$ 
we define
$[\det H^n(X)]\in
H^1(U,{\Z}/2{\Z})$.
Let now $k$ be a field
and let $f\in S^dE\otimes k$ be a homogeneous
polynomial of degree $d$
defining a smooth hypersurface $Y$
in ${\P}^{n+1}_k$.
Then,
the pull-back in $H^1(k,
{\Z}/2{\Z})
={\rm Hom}(\Gamma_k^{ab},
{\Z}/2{\Z})$ of
$[\det H^n(X)]$
by the $k$-valued point of $U$
corresponding to $f$
is given by the determinant
of the orthogonal representation
$H^n(Y_{\bar k},{\Q}_\ell
(\frac n2))$
for a prime number $\ell$
invertible in $k$.

\begin{thm}\label{thmsign}
Let $n\geqslant 0$ and $d > 1$
be integers; assume that  $n$  is even.
Define the sign 
$\varepsilon(n,d)=\pm1$
by \begin{equation}
\varepsilon(n,d)
=
\begin{cases}
(-1)^{\frac{d-1}2} 
&\text{ if $d$ is odd,}\\
(-1)^{\frac d2\frac {n+2}2}
&
\text{ if $d$ is even.}
\end{cases}
\label{eqsign}
\end{equation}
Then, we have
\begin{equation}
[\det H^n(X)]
=[\varepsilon(n,d)\cdot
{\rm disc}_d(F)]
\label{eqpm}
\end{equation}
in $H^1(U_{{\Z}[\frac12]},{\Z}/2{\Z})$.
\end{thm}

By a standard specialization
argument,
Theorem \ref{thmsign}
implies the Theorem
stated in the introduction.

\smallskip
\begin{proof}
The Kummer sequence defines
an exact sequence

\begin{equation}
0\to \Gamma(U_{\frac12},
{\mathcal O})^\times
/(\Gamma(U_{\frac12},
{\mathcal O})^\times)^2
\to
H^1(U_{\frac12},
\Z/2\Z)
\to
\Pic(U_{\frac12})[2]
\to 0,
\label{eqPic}
\end{equation}

\n where we have written $U_{\frac12}$ instead of $U_{\Z[\frac12]}$ for typographical reasons, and $\Pic(U_{\frac12})[2]$ denotes the subgroup of $\Pic(U_{\frac12})$ killed by $2$.

\smallskip

We now compute 
$\Gamma(U_{\frac12},
{\mathcal O})^\times$
and 
${\rm Pic}(U_{\frac12})$.
In the exact sequence

\begin{equation}
0\to \Gamma(\P_{{\Z}[\frac12]},
{\mathcal O})^\times \! \to
\Gamma(U_{\frac12},
{\mathcal O})^\times \!
\to {\Z}
\to \!
{\rm Pic}(\P_{{\Z}[\frac12]})
\to\!
{\rm Pic}(U_{\frac12})\to0,
\label{eqPicU}
\end{equation}
\n the Picard group
${\rm Pic}(\P_{{\Z}[\frac12]})$
is canonically identified with
${\Z}$ by
the generator $[{\mathcal O}(1)]$.
Then, the map
${\Z}
\to
{\rm Pic}(\P_{{\Z}[\frac12]})$
is identified with the multiplication
by $m
=(n+2)(d-1)^{n+1}$
since it sends 1 to
the class of ${\mathcal O}(m)$.
Thus, we have

\smallskip

$\Gamma(U_{\frac12},
{\mathcal O})^\times
=
\Gamma(\P_{{\Z}[\frac12]},
{\mathcal O})^\times
= {\Z}[\frac12]^\times
=\langle-1,2\rangle$.

\smallskip

\n It also follows from (\ref{eqPicU})
that
the Picard group
${\rm Pic}(U_{\frac12})$
is cyclic of order $m$; since $m$ is even, this shows that ${\rm Pic}(U_{\frac12})[2]$ has order 2. .

Let $\bar \xi$ be
a geometric generic point of
the irreducible divisor $D$
and let
$I_{\bar \xi}$ denote the
absolute Galois group
of the fraction field of
the strict henselization
${\mathcal O}_{\P,\bar \xi}$.
Since the profinite group
$I_{\bar \xi}$ is isomorphic
to $\widehat {\Z}$,
the group
${\rm Hom}(I_{\bar\xi},{\Z}/2{\Z})$
is of order 2.

We show that
the images of
$[\det H^n(X)]$ and
$[{\rm disc}_d(F)]$
by the restriction map
$H^1(U_{\frac12},
{\Z}/2{\Z})
\to
{\rm Hom}(I_{\bar\xi},{\Z}/2{\Z})$
are both the unique non-trivial element.
For $[{\rm disc}_d(F)]$,
this follows from
that ${\rm disc}_d(F)$
is a prime element of 
the irreducible divisor $D$.

Let
 $\bar \eta$ denote
a geometric generic point of
${\rm Spec}\ {\mathcal O}_{\P,\bar \xi}$.
We show that
the character $\det
H^n(X_{\bar \eta},{\Q}_\ell)$
of $I_{\bar \xi}$ is 
the unique non-trivial
character of order 2.
By the result of  \cite{SGA7} recalled 
in  Lemma \ref{lmSGA7}.2
applied to $k=\bar {\Q}$,
the geometric fiber $X_{\bar \xi}$
has a unique singular
point which is 
an ordinary quadratic singularity
in $X_{\bar \xi}$.
Hence, by the Picard-Lefschetz formula
\cite[Th\'eor\`eme 3.4 (ii)]{PL},
we have an exact sequence
\begin{align}
0\to
&H^n(X_{\bar \xi},{\Q}_\ell)
\to
H^n(X_{\bar \eta},{\Q}_\ell)
\to
{\Q}_\ell(\textstyle{\frac n2})\label{PL}
\\
&
\to
H^{n+1}(X_{\bar \xi},{\Q}_\ell)
\to
H^{n+1}(X_{\bar \eta},{\Q}_\ell)
\to 0
\nonumber
\end{align}
of $\ell$-adic representations
of the inertia group $I_{\bar \xi}$.
Further, since $X$ is regular,
the base change $X_{{\mathcal O}_{\P,\bar \xi}}$
is also regular and 
the inertia group $I_{\bar\xi}$
acts on 
${\Q}_\ell(\frac n2)$
via the unique non-trivial character
$I_{\bar\xi}\to \{\pm1\}$
by
\cite[Th\'eor\`eme 3.4 (iii)]{PL},
\cite{IPL}.
Since $I_{\bar \xi}$
acts trivially on
$H^{n+1}(X_{\bar \xi},{\Q}_\ell)$
and on
$H^n(X_{\bar \xi},{\Q}_\ell)$,
the boundary map
${\Q}_\ell(\frac n2)
\to
H^{n+1}(X_{\bar \xi},{\Q}_\ell)$
in (\ref{PL}) is the zero-map
and
the character
$\det H^n(X_{\bar \eta},{\Q}_\ell)$
of $I_{\bar \xi}$ is non-trivial.

The composition map

\smallskip

$\Gamma(U_{\frac12},
{\mathcal O})^\times
/(\Gamma(U_{\frac12},
{\mathcal O})^\times)^2
\to
H^1(U_{\frac12},
{\Z}/2{\Z})
\to
{\rm Hom}(I_{\bar\xi},{\Z}/2{\Z})$

\smallskip

\n is 0
since the strict henselization
${\mathcal O}_{\P,\bar \xi}$
containes $\bar {\Q}$
as a subfield.
By (\ref{eqPic}) we thus get 
a map ${\rm Pic}(U_{\frac12})
[2]
\to 
{\rm Hom}(I_{\bar\xi},{\Z}/2{\Z})$.

\smallskip

Since the images of
$[\det H^n(X)]$ and
$[{\rm disc}_d(F)]$
in 
${\rm Hom}(I_{\bar\xi},{\Z}/2{\Z})$
are non-trivial,
the map
${\rm Pic}(U_{\frac12})[2]
\to
{\rm Hom}(I_{\bar\xi},{\Z}/2{\Z})$
is an isomorphism of groups of order 2.
Further by (\ref{eqPic}),
the difference
$[{\rm disc}_d(F)]-
[\det H^n(X)]$
is in the image of the map

\ $\Gamma(U_{\frac12},
{\mathcal O})^\times \!
=
\! {\Z}[\frac12]^\times \ \
\to
\ \ H^1(U_{\frac12},
{\Z}/2{\Z})$.

\ Therefore,
$[\det H^n(X)]$
equals either
$[\pm \, {\rm disc}_d(F)]$
or $[\pm2 \, {\rm disc}_d(F)]$.
We show that the latter case
is not possible.
Let $K$ be the local field of 
${\P}$ at the generic point
of the fiber ${\P}_{{\F}_2}$.
Then, the character
$[\det H^n(X)]$
induces an unramified character
of $\Gamma_K$.
On the other hand,
 $[\pm 2 \, {\rm disc}_d(F)]$
corresponds to a totally ramified
quadratic extension of $K$.
Hence the latter case is excluded
and we obtain
$[\det H^n(X)]=
[\pm \, {\rm disc}_d(F)]$.

It remains to show that the sign is $\varepsilon(n,d)$. To do so, let
${\rm Spec}\
{\R}
\to U$
be the map defined
by the 
homogenous polynomial
$f=T_0^d+\cdots+
T_{n+1}^d$.
Then, since
${\rm disc}_d(f)
=d^{-a(n,d)}{\rm disc}_r(f)>0$,
the pull-back of
$[{\rm disc}_d(F)]$
in 
$H^1({\R},
{\Z}/
2{\Z})$
is $1$.
On the other hand,
by Corollary \ref{cortop}
and by Artin's comparison
theorem relating singular
and \'etale cohomology (\cite[Th\'eor\`eme 4.1]{A}),
the pull-back of
$[\det H^n(X)]$
equals 
$\varepsilon(n,d)
\in H^1({\R},
{\Z}/
2{\Z})
=\{\pm 1\}$.
Thus the assertion
is proved.
\end{proof}

\section{Discriminant modulo 4}\label{s2}

In this section, we keep assuming that
$n$ is even.

Let us first prove the following elementary fact:

\smallskip 

\begin{lm}\label{lm2}
Let $K$ be a complete discrete valuation
field such that $2$ is a uniformizer.
Let $u\in {\mathcal O}_K^\times$
be a unit which is not a square
and let
$L$ denote the quadratic extension
$K(\sqrt u)$.

{\rm 1}.
The extension $L$
is unramified over $K$
if and only if there exists
a unit $v\in {\mathcal O}_K^\times$
such that
$u\equiv v^2 \pmod 4$.

{\rm 2.}
Assume that the extension $L$
is unramified over $K$.
Then, for every unit $v$
satisfying 
$u\equiv v^2 \pmod 2$,
we have
$u\equiv v^2 \pmod 4$.
Further,
the corresponding residue field extension
is given by the Artin-Schreier equation
$t^2+t=w$, where $w$ is the image of $\frac{1}{4}(uv^{-2}-1)$ in the residue field.
\end{lm}

\n [Recall that ``unramified''  implies that the residue extension
is separable.]

\begin{proof}
Let $F$ be the residue field of $K$.
If $\bar u\in F^\times$ is not a square,
the residue field of $L$ 
is a quadratic radicial extension
of $F$.
Hence, by dividing $u$ by a square,
we may assume $u\equiv1\pmod 2$.
We put $u=1+2a$. 
Substituting $x=1+2t$
into the equation $x^2=u$,
we obtain $t^2+t=a/2$. 
If $a$ is a unit,
a solution $t$ cannot be an integer
and 
we have $2 \, {\rm ord}_L t=
-{\rm ord}_L 2$.
Thus $L$ is totally ramified over $K$.
If $a=2b$,
the extension $L$ is unramified
and the residue extension
is given by $t^2+t=\bar b$.

If $v^2\equiv v^{\prime 2}\pmod 2$,
we have $(v/v')^2\equiv 1\pmod 4$
and assertion 2 follows.
\end{proof}

\begin{thm}\label{thm4}
Let $n\geqslant 0$ and $d > 1$
be integers. 
We assume that $n$ is even
and define the sign
$\varepsilon(n,d)= \pm1$
by {\rm (\ref{eqsign})}.
Then, 
$\varepsilon(n,d)\cdot
{\rm disc}_d(F)$
is congruent to a square modulo $4$ in 
$S^m((S^dE)^\vee)
=\Gamma(\P,{\mathcal O}(m))$, where
 $m=(n+2)(d-1)^{n+1}$.
More precisely, 
there exist homogeneous polynomials
$A\in S^{\frac m2}((S^dE)^\vee)$
and
$B\in S^m((S^dE)^\vee)$
satisfying
$\varepsilon(n,d)\cdot
{\rm disc}_d(F)
=A^2+4B$.
\end{thm}

\smallskip

\begin{proof}
Let $K$ be the local field of 
${\P}$ at the generic point $\xi$
of the fiber ${\P}_{{\F}_2}$.
Namely, $K$ is the fraction field of the completion
of the local ring ${\mathcal O}_{\P,\xi}$.
The residue field $F=\kappa(\xi)$
is the function field of $\P_{{\F}_2}$.
Take a global section $A_1\in \Gamma(\P,
{\mathcal O}(\frac m2))$
not divisible by $2$.
The germ of $A_1$ is a basis of
the stalk of ${\mathcal O}(\frac m2)$
at $\xi$.
Since the germ of $\varepsilon(n,d)\cdot
{\rm disc}_d(F)$ 
is also a basis of the stalk 
of ${\mathcal O}(m)$
at $\xi$,
the ratio
$\varepsilon(n,d)\cdot
{\rm disc}_d(F)/A_1^2$ is a unit
at $\xi$. 

By Theorem \ref{thmsign} we have

\begin{equation}
[\det H^n(X)]
=
[\varepsilon(n,d)\cdot
{\rm disc}_d(F)]
\label{eqpm2}
\end{equation}

\noindent in $H^1(U_{\Z[\frac 12]},{\Z}/2{\Z})$.
Since $[\det H^n(X)]$ is the
restriction of an element of
$H^1(U,{\Z}/2{\Z})$,
the extension of $K$ 
generated by the square root of
$\varepsilon(n,d)\cdot
{\rm disc}_d(F)/A_1^2$
is an unramified extension.
Hence by Lemma \ref{lm2}.1,
there exists a unit $v\in 
{\mathcal O}_K^\times$ such that
$\varepsilon(n,d)\cdot
{\rm disc}_d(F)
\equiv v^2\cdot A_1^2
\pmod4$.

We consider the local section
$\bar A=v\cdot A_1
\pmod 2$ of 
${\mathcal O}_{\P_{{\F}_2}}(m)$.
Since its square has no pole,
the local section 
$\bar A$ has the same property and it
defines a global section
$\Gamma(\P_{{\F}_2},{\mathcal O}(m))$.
Let us choose a lifting
$A\in 
\Gamma(\P,{\mathcal O}(m))$
of $\bar A$.
Since $\varepsilon(n,d)\cdot
{\rm disc}_d(F)/A^2
\equiv 1
\pmod2$,
we have $\varepsilon(n,d)\cdot
{\rm disc}_d(F)/A^2
\equiv 1
\pmod4$
by Lemma \ref{lm2}.2.
Namely,
the difference $\varepsilon(n,d)\cdot
{\rm disc}_d(F)
- A^2$ is divisible by $4$
at $\xi$ and hence divisible
on ${\P}$.
\end{proof}

\smallskip
\begin{cor}\label{cor2}
Let $k$ be a field of
characteristic $2$
and $Y$ be a smooth
hypersurface of
even dimension $n$
in ${\P}^{n+1}_k$
defined by a homogeneous
polynomial $f$ of degree $d$.
Let $A(f)\in k^\times$
and $B(f)\in k$ 
denote the specialization
of the polynomials
$A$ and $B$ occurring in Theorem \ref{thm4}.1.
Then, the quadratic
character
$\det H^n(Y_{\bar k},
{\Q}_\ell(\frac n2))$
of $\Gamma_k$
is defined by
the Artin-Schreier
equation $t^2+t=B(f)\cdot A(f)^{-2}$.
\end{cor}
\smallskip

\noindent [In other words, the kernel of $\det V : \Gamma_k \to \{\pm1\}$ 
is the
subgroup of $\Gamma_k$ corresponding to the 
Artin-Schreier extension of $k$
defined by 
$t^2+t=B(f)\cdot A(f)^{-2}$.]

\smallskip

\begin{proof}
Since the divided discriminant
${\rm disc}_d(F)$
is a basis of ${\mathcal O}_U(m)$,
the section $A$ in Theorem \ref{thm4}
is a basis of ${\mathcal O}_{U_{{\F}_2}}
(\frac m2)$.
The pull-back of 
$[\det H^n(X)]$ in 
$H^1(U_{{\F}_2},
{\Z}/2{\Z})$
is defined by
$t^2+t=B\cdot A^{-2}$
by Theorem \ref{thm4}
and Lemma \ref{lm2}.
Hence the assertion follows 
by specialization.
\end{proof}

\section{Examples of Discriminants}\label{sEx}

\subsection{Binary forms}
Let $F=C_0T_0^d
+C_1T_0^{d-1}T_1+
\cdots +C_dT_1^d$ 
be the universal binary
polynomial of degree
$d > 1$ defining the universal 
family of finite schemes
of degree $d$ in ${\P}^1$.
The divided discriminant
${\rm disc}_d(F)$
is a homogeneous polynomial
in $(C_i)$
of degree $m=2d-2$
and the sign $\varepsilon(0,d)$
is $(-1)^{d(d-1)/2}$.

The signed discriminant
$\varepsilon(n,d)\cdot
{\rm disc}_d(F)$
is equal to
${\rm dis}_d(F)=
\tilde\Delta(C_0,\ldots,C_d)$
in the notation of
\cite[Chap.\ 4, \S6, n$^{\rm o}$ 7, formula (52)]{Bo}
where the subscript $d$ stands for the degree.
For $d=2$,
we have $\varepsilon(0,2)=-1$
and $\varepsilon(0,2)\cdot
{\rm disc}_d(F)=b^2-4ac$
for $(a,b,c)=(C_0,C_1,C_2)$.

The discriminant
${\rm disc}_d(F)
\in {\Z}
[C_0,\ldots,C_d]$
is the unique
polynomial such that,
if $F=
\prod_{i=1}^d
(u_iT_0-v_iT_1),$
then 
\begin{equation}
{\rm disc}_d(F)
=
\prod_{i\neq j}
(u_iv_j-u_jv_i)
\label{eqbin}.
\end{equation}
Indeed, 
the discriminant
${\rm disc}_d(F)$ is
divisible in the
polynomial ring
${\Z}[u_1,\ldots,u_d,
v_1,\ldots,v_d]$
by the right
hand side
and hence is a constant
multiple of it
since the degrees are equal.
By evaluating at
$u_i=1$ and $v_i=\zeta_d^i$
for a primitive $d$-th root
$\zeta_d$ of $1$,
we get
$F=T_0^d-T_1^d$
and $d^{d-2}
\cdot {\rm disc}_dF=
{\rm disc}_rF=
{\rm res}(dT_0^{d-1},
-dT_1^{d-1})
=(-1)^{d-1}d^{2(d-1)}$.
For the right hand side,
we have
$\prod_{i\neq j}
(\zeta_d^j-\zeta_d^i)=
\zeta_d^{\binom d2(d-1)}
\prod_{i=1}^{d-1}
(1-\zeta_d^i)^d=
(-1)^{d-1}d^d$
and (\ref{eqbin}) follows.

Hence, the signed discriminant
is given by
\begin{equation}
(-1)^{d(d-1)/2}\cdot
{\rm disc}_d(F)
=
\prod_{i< j}
(u_iv_j-u_jv_i)^2;
\label{eqbin2}
\end{equation}
it is congruent to
the square of
$\prod_{i< j}
(u_iv_j+u_jv_i)$
in ${\Z}
[C_0,\ldots,C_d]$
modulo $4$.

\subsection{Quadrics}
Let $F
=\sum_{0\leqslant i\leqslant 
j\leqslant n+1}C_{ij}T_iT_j$
be the universal quadratic
polynomial defining the universal
family of quadrics
of even dimension $n\geqslant 0$.
The divided discriminant
${\rm disc}_d(F)$ is the same as the resultant discriminant ${\rm disc}_r(F)$;
it is a homogeneous polynomial
in $(C_{ij})$
of degree $m=n+2$
and the sign $\varepsilon(n,2)$
is $(-1)^{(n+2)/2}$.

Let  $A=(A_{ij})$ be the $m \times m$ symmetric matrix
with coefficients
in $\Z[C_{ij};\ 0\leqslant i\leqslant
j\leqslant n+1]$ defined
by $A_{ij}=A_{ji}=C_{ij}$
for $i< j$ and
$A_{ii}=2C_{ii}$.
We have 
$TA\!\ ^tT=2F$
where $T$
is the row vector
$\begin{pmatrix}
T_0 & \ldots& T_{n+1}
\end{pmatrix}$
and \begin{equation}
{\rm disc}_d(F)= {\rm disc}_r(F)
=
\det A.
\label{eqQ}
\end{equation}
Indeed, by inverting $2$,
we see that
$\det A$ equals the divided discriminant
${\rm disc}_d(F)$
up to a constant $\lambda$ which is $\pm$ a power of 2.
For the unit quadratic form
$\sum_{i=0}^{n+1}T_i^2$,
we have $ {\rm disc}_r(F) = 2^{n+2} = \det A$ hence $\lambda = 1$. This proves (\ref{eqQ}).

Let $k$ be a field
of characteristic $\neq2$ and let
$Q$ be a smooth quadric
of even dimension $n$
in ${\P}^{n+1}_k$
defined by a non-degenerate
quadratic form $q(x)=\
\!\!^t\! xBx$
for a symmetric matrix
$B\in M_{n+2}(k)$.
Then, the
character $\det
H^n(Q_{\bar k},{\Q}_\ell
(\frac n2))$
is defined by the
square root of
the signed discriminant
$\varepsilon(n,2) 
{\rm disc}_d\ q
=(-1)^{\frac{n+2}2}
2^{n+2}\det B$;
this follows from 
(\ref{eqQ}) applied to $A=2B$.

\subsection{Plane cubics}
Let $C\to \P={\P}((S^3E)^\vee)$ 
be the universal
family of cubic curves
defined by the universal cubic
polynomial $F$
for $n=1$ and $d=3$.
The divided
discriminant ${\rm disc}_d(F)$
may be viewed as a section of
$\Gamma(\P,{\mathcal O}(12))$.
We compare it with the discriminant
of the Jacobian
defined by a Weierstrass equation.

Let $p\colon
{\mathbb P}\to \P$
denote the ${\P}^2$-bundle
${\P}({\mathcal E}^\vee)
={\rm Proj}S^\bullet
{\mathcal E}^\vee$
defined by the dual
of ${\mathcal E}
=
{\mathcal O}(2)
\oplus
{\mathcal O}(3)
\oplus
{\mathcal O}$
and let
$X\in \Gamma({\mathbb P},
p^*{\mathcal O}(2)
\otimes 
{\mathcal O}(1)),
Y\in \Gamma({\mathbb P},
p^*{\mathcal O}(3)
\otimes 
{\mathcal O}(1)),
Z\in \Gamma({\mathbb P},
{\mathcal O}(1))$
denote the components
of the tautological map
$p^*{\mathcal E}^\vee
\to {\mathcal O}(1)$.
For $i=1,2,3,4,6$,
homogeneous polynomials
$a_i(F)\in S^i(S^3E)^\vee
=\Gamma(\P,{\mathcal O}(i))$ 
of degree $i$
of the coefficients of $F$
are defined 
in \cite[(1.6)]{AVT}.
Let $W$ be the
closed subscheme of ${\mathbb P}$
defined by the homogeneous
Weierstrass equation
\begin{equation}
E\colon Y^2Z+a_1XYZ+a_3YZ^2 -(
X^3+a_2X^2Z+a_4XZ^2+a_6Z^3) = 0.
\label{eqW}
\end{equation}
Then 
the functor
${\rm Pic}^0_{C/\P}$
is represented by 
the smooth locus $J$ of $W$
\cite[Theorem 1]{AVT}.

Let ${\rm disc}_e(E)
\in \Gamma(\P,{\mathcal O}(12))$
denote
the discriminant of the Weierstrass
equation (\ref{eqW})
in the standard sense
defined as a polynomial of
$a_i(F)$, cf.\ \cite[(1.3)]{AVT}.
Then, we have
\begin{equation}
{\rm disc}_d(F)=- \, 
{\rm disc}_e(E).
\label{eqE}
\end{equation}
Indeed,
since both
${\rm disc}_d(F)$
and
${\rm disc}_e(E)$
are bases on
the smooth locus
$U=\P \! \sm \! D$
and $\Gamma(U,{\mathcal O}^\times)
={\Z}^\times
=\{\pm 1\}$,
we have
${\rm disc}_d(F)=
\pm \,
{\rm disc}_e(E)$.
We determine the sign
by testing it on the Fermat cubic
$f=x^3+y^3+z^3$.
By (\ref{eqadn}),
we have
${\rm disc}_d(f)=
3^{-3}{\rm disc}_r(f)=
3^{-3}\cdot 3^{4\cdot 3}=3^9$.
On the other hand,
since $c_4$ and $c_6$ 
(\cite[(1.3)]{AVT})
are given by
$c_4=0$ and
$c_6=2^3\cdot 3^6$,
we obtain
${\rm disc}_e(E)=
- \, 3^9$.

\subsection{Cubic surfaces}

We put $A={\Z}
[a,b,c,d,e]$
and ${\mathbb P}
={\rm Proj}\ A$.
Define a closed
subscheme $S$ of
${\P}^4
\times {\mathbb P}
={\rm Proj}\
{\Z}[x,y,z,u,v]
\times {\mathbb P}$
by the Sylvester-type equations
\begin{equation}
x+y+z+u+v=0,\
ax^3+by^3+cz^3+du^3+ev^3=0.
\label{eqSil}
\end{equation}
It is known that
a general enough cubic
equation in characteristic $0$
can be put in
that form after a suitable finite extension
of the ground field,
cf.\ \cite[\S 9.4.1]{Dol}.
In 1862, Salmon (\cite[\S543]{Sa}) showed that, over {\bf C}, 
the corresponding cubic surface is smooth
if and only if a certain polynomial 
${\rm disc}_s(a,b,c,d,e)
\in A$ is non-zero. This polynomial is homogeneous of degree $32$
in $a,b,c,d,e$. Its definition is:
\begin{equation}
{\rm disc}_s(a,b,c,d,e)= ((s^2-64rt)^2-4t^3p)^2 - 2^{11}(8t^6q+t^4s(s^2-4rt)),
\label{eqSe}
\end{equation}
\noindent where \ $p=a+b+c+d+e,\ q = ab + \cdots,  \ r = abc + \cdots,  \ s = abcd + \cdots, \ t=abcde$, \ are the elementary
symmetric functions of  $a,b,c,d,e$. [We give here the corrected 
version of the formula due to W.L. Edge, cf.\ \cite{Ed} and \cite {CS}; there were numerical mistakes in \cite{Sa}.] 

By eliminating
one variable in
(\ref{eqSil}),
one obtains a cubic
polynomial $F_s$ with
coefficients in $A$ and 
its divided discriminant
${\rm disc}_d(F_s)$ is a well-defined element of $A$. The relation between the Salmon discriminant and the divided one is:

\begin{equation}
  {\rm disc}_s(a,b,c,d,e) = 3^{-27}{\rm disc}_d(F_s).
\label{ss}
\end{equation}

\smallskip

Indeed,
the polynomial 
${\rm disc}_s(a,b,c,d,e)$
is geometrically irreducible
\cite[Lemma 2.5]{CS}
and 
the smooth locus of
$S_{\Q}\to 
{\mathbb P}_{\Q}$ is
defined by 
${\rm disc}_s(a,b,c,d,e)
\neq 0$, cf.\ \cite{Sa}, loc.cit.,  and 
\cite[Corollary 2.10]{CS}.
This implies
that 
${\rm disc}_s(a,b,c,d,e)$
is a constant 
multiple
of ${\rm disc}_d(F_s)$.
We determine the constant
by testing it on the Fermat cubic
$$f=y^3+z^3+u^3+v^3$$
corresponding to
$(a,b,c,d,e)
=(0,1,1,1,1)$.

By (\ref{eqSe}), we have
${\rm disc}_s(0,1,1,1,1)
=1$.
On the other hand,
we have
${\rm disc}_d(f)=
3^{-5}{\rm disc}_r(f)
=3^{-5}3^{32}=3^{27}$.
This proves formula (\ref{ss}).

Since
${\rm disc}_s(1,1,1,1,1)
=-3^55$
for the Clebsch surface
corresponding to
$(a,b,c,d,e)
=(1,1,1,1,1)$,
the equality (\ref{ss})
implies 
${\rm disc}_d(f)=
-3^{32}5$
for the corresponding
polynomial $f =  \sum x^3 \ - \ (\sum x)^3$. Since $f$ is divisible by $3$, this means that ${\rm disc}_d(f/3)= -5$, which fits with the fact that $f/3$ defines a smooth cubic in every characteristic except $5$.
See also 
\cite[Remark 2.11]{CS}.

Let $S\subset {\P}^3_k$
be a smooth cubic surface
defined by a cubic form $f\in
S^3E_k$.
Then
$H^2(S_{\bar k},
{\Q}_\ell(1))$
is spanned by the classes
of the 27 lines 
\cite[p.588]{We}, \cite[\S27]{Ma}.
The group of automorphisms of
the ${\Z}$-lattice spanned by
the classes of these
lines permuting them and
preserving the
intersection form
is isomorphic 
to the Weyl group $W(E_6)$
of the root system $E_6$.
The kernel of the
 determinant map
$W(E_6)\to \{\pm1\}$
is a simple group
of order 25920.

The action of $\Gamma_k$
on the 27 lines
defines a homomorphism
$$\Gamma_k \to W(E_6),$$
unique up to conjugation.
By applying Theorem \ref{thmsign}
and using specialization,
we see that {\it the composition
$\Gamma_k\to W(E_6)\to \{\pm1\}$
is defined by
the square root
of the signed discriminant}
$- \! {\rm \ disc}_d f$.
Note that for a cubic surface, 
we have $n=2, d=3$ so that $\varepsilon(n,d) = -1$.
In the case where
the characteristic of $k$
is not 3,
this is equivalent to
\cite[Theorem 2.12]{CS}
which is stated in terms of the Salmon discriminant
${\rm disc}_sf$.

The formula (\ref{eqSe})
together with (\ref{ss})
and $\varepsilon(2,3)=-1$
implies the congruence
$\varepsilon(2,3)
{\rm disc}_d(F_s)\equiv
- 3 s^4\pmod 8$.
Consequently,
for a cubic surface
in Sylvester form
(\ref{eqSil})
in characteristic $2$,
the determinant map $\Gamma_k \to  \{\pm1\}$ is trivial if and only if 
$k$ contains ${\F}_4$.

\end{document}